\documentclass[11pt]{article}
\usepackage{amsmath,amssymb,amsthm,amscd,dsfont}
\usepackage{amsfonts,latexsym,rawfonts,amsmath,amssymb,amsthm}
\usepackage{lscape}
\usepackage{amscd, float,times,rotating}
\usepackage{pb-diagram}
\usepackage{hyperref}
\usepackage[pagewise]{lineno}

\numberwithin{equation}{section}

\def\g{{\bar{g}}}
\def\M{{\partial M}}

\def\R{{\bar{R}}}
\def\A{{\bar{A}}}
\def\D{{\bar{D}}}
\def\S{{\bar{S}}}
\def\H{{\bar{H}}}
\def\F{{\bar{F}}}
\newtheorem{prop}{Proposition}[section]
\newtheorem{theo}[prop]{Theorem}
\newtheorem{lemm}[prop]{Lemma}

\newtheorem{rema}[prop]{Remark}

\def\begeq{\begin{equation}}
\def\endeq{\end{equation}}

\begin{document}
\title{Finite Boundary Regularity for Conformally Compact Einstein Manifolds of Dimension 4}
\author{Xiaoshang Jin \thanks{The author’s research is supported by China Postdoctoral Science Foundation (Grant No.2019M650287)}}
\date{}
\maketitle
\begin{abstract}
We prove that a $4-$dimensional $C^2$ conformally compact Einstein manifold with H\"older continuous scalar curvature and with $C^{m,\alpha}$ boundary metric has a $C^{m,\alpha}$ compactification. We also study the regularity of the new structure and the new defining function. This is a supplementary proof of Anderson's work and an improvement of Helliwell's result in dimension 4.
\end{abstract}
\section{Introduction}  \label{sect1}
\par In 1985, Charles Fefferman and Robin Graham \cite{graham1985conformal} introduced a new method to study the local conformal invariants of manifolds. Similar to $n-$ sphere embedded into $n+2-$dimensional Minkowski space, they tried to embed an arbitrary conformal $n-$manifold into an $n+2-$ dimensional Ricci-flat Lorentz manifold, which they called the ambient space. The ambient spaces were used to produce local scalar conformal invariants. An important part of the ambient space construction is the introduction
of conformally compact Einstein metrics for a conformal manifold. The study of conformally compact Einstein metrics could tell us some relationship between the Riemannian structure in the interior and the conformal structure on the boundary. Much progress has been made since then. In recent years, the physics community has also become interested in
conformally compact Einstein metrics because the introduction of AdS/CFT correspondence in the quantum theory of gravity in theoretic physics by Maldacena \cite{maldacena1999large}.
\par Let $M$ be the interior of a compact $(n+1)$-dimensional manifold $\overline{M}$ with non-empty boundary $\M$. We call a complete metric $g^+$ on $M$ is $C^{m,\alpha}$(or $W^{k,p}$) conformally compact if there exits a defining function $\rho$ on $\overline{M}$ such that the conformally equivalent metric $$g=\rho^2g_+$$ can extend to a $C^{m,\alpha}$(or $W^{k,p}$) Riemannian metric on $\overline{M}.$ The defining function is smooth on $\overline{M}$ and satisfies
 \begin{equation}\label{1.1}
 \left\{
    \begin{array}{l}
    \rho>0\ \ in \ M
    \\\rho=0\ \ on\ \M
    \\d\rho\neq 0\ \ on\  \M
    \end{array}
 \right.
 \end{equation}
 Here $C^{m,\alpha}$ and $W^{k,p}$ are usual H\"older space and the Sobolev space. We call the induced metric $h=g|_{\M}$  the boundary metric associated to the compactification $g.$ It is easy to see that different defining function induces different boundary metric and every two of the boundary metrics are conformal equivalent. Then  the conformal class $[h]$ is uniquely determined by $(M,g_+).$ We call $[h]$ the conformal infinity of $g_+.$  If in addition, $g^+$ is Einstein, i.e.
  \begin{equation}\label{1.2}
Ric_{g_+}+ng_+=0,
 \end{equation}
 then we say $(M,g_+)$ is a conformally compact Einstein manifold.
 \par There are some interesting problems concerning conformally compact Einstein metric. Such as the existence problem, see \cite{anderson2008einstein}, \cite{gursky2017non} ,\cite{graham1991einstein}, \cite{gursky2020local}, \cite{kichenassamy2004conjecture}, \cite{lee2006fredholm} etc. The unique problem, see \cite{anderson2003boundary},\cite{chang2018compactness2}. The compactness problem, see \cite{anderson2008einstein},\cite{chang2018compactness1}, \cite{chang2018compactness2}.
 \par In this paper, we deal with the boundary regularity problem. Given a conformally compact Einstein manifold $(M,g^+)$ and a compactification $g=\rho^2g^+,$ if the boundary metric $h$ is  $C^{m,\alpha},$ is there a $C^{m,\alpha}$ compactification of $g^+?$
This problem  was first raised by Fefferman and Graham in 1985 in \cite{graham1985conformal} and they observed that if $dim M=n+1$ is odd, the boundary regularity in general breaks down at the order $n.$ If $dim M=n+1$ is even, the $C^{m,\alpha}$ compactification may exist.
\par In \cite{chrusciel2005boundary}, Chru\'sciel, Delay, Lee and Skinner used the harmonic diffeomorphism at infinity to construct a good structure near boundary where Einstein equation could be written as an elliptic PDE of second order uniformly degenerating at the boundary. That is the so-called 'gauge-broken Einstein equation'. Then they use polyhomogeneity result of some specific degenerate equation to obtain a good result of the boundary regularity. We suggest the readers to see \cite{andersson1996solutions} for more details about these equations. They proved that if the boundary metrics are smooth, the $C^2$ conformally compact Einstein metrics have conformal compactifications that are smooth up to the boundary in the sense of $C^{1,\lambda}$ diffeomorphism in dimension 3 and all even dimensions, and polyhomogeneous smooth in odd dimensions greater than 3. This is certainly a very good result in the sense that they made good use of Einstein equation and gave us a suitable coordinate in infinity to study conformally compact Einstein metrics. I think their method is more geometrical. The condition of that the initial compactification is $C^2$ in all dimension should be sharp. However, their result only hold for smooth case. It is believed that their method could also be used to prove the finite regularity although we may loss half regularity in this situation.
 \par In \cite{anderson2003boundary} and \cite{anderson2008einstein}, M. T. Anderson considered the Bach tensor in dimension 4, and proved the finite regularity result. He only assume that the initial compactification $g$ is $W^{2,p}$ where $p>4.$ I am not sure whether the $W^{2,p}$ condition is good enough to prove his result. As a supplementary  proof, we use Anderson's method to prove his conclusion where we assume that the initial compacification $g$ is $C^2$ and the scalar curvature is $C^\sigma$ for some $\sigma\in(0,1).$
\par In \cite{helliwell2008boundary}, Helliwell solved the issue in all even dimensions by following Anderson's method. He considered the Fefferman-Graham ambient obstruction tensor instead of Bach tensor in higher dimensions. It is conformally invariant and vanishes for Einstein metrics.  Helliwell assumed the initial compactification $g$ is at least in $C^{n,\alpha}$ for a $(n+1)-$ smooth manifold. It means the original compactification is $C^{3,\alpha}$ for a smooth manifold of dimension 4. Now we reduce the condition $C^{3,\alpha}$ to $C^{2,\sigma}$ to improve his result.
\par This is the main result:
\begin{theo}\label{theorem 1.1}
Let $(M,g^+)$ be a conformally compact Einstein manifold of dimension $4$ with a $C^2$ compactification $g=\rho^2g^+.$ If the scalar curvature $S\in C^\sigma(\overline{M})$ for some $\sigma>0,$ the boundary metric $h=g|_{\partial M}\in C^{m,\alpha}(\partial M)$ with $m\geq 2,\alpha\in(0,1),$ then under a $C^{2,\lambda}$ coordinates change, $g^+$ has a $C^{m,\alpha}$ conformally compactification $\tilde{g}=\tilde{\rho}^2g^+$ with the boundary metric $\tilde{g}|_{\M}=h.$
\end{theo}
\begin{rema}
The new coordinates form $C^{m+1,\alpha}$ differential structure of $\overline{M}.$
$\tilde{\rho}$ is a $C^{m+1,\alpha}$  defining function.
\par If $g=\rho^2g^+$ is $C^{2,\sigma},$ then the condition of $S$ in theorem 1.1 holds automatically. Hence the conclusion is also true.
\par If the boundary metric $h$ is smooth, then $g^+$ has a smoothly conformally compactification $\tilde{g}$ with the boundary $\tilde{g}|_{\M}=h.$
\par The condition that 'the scalar curvature $S\in C^\sigma(\overline{M})$' seems unnatural and this is because we choose the Yamabe compactification for the new $\tilde{g}.$ This condition is used to improve the regularity of the new defining function and new compactification for the Yamabe equation with Dirichlet data. I think the condition may be removed if we choose another 'good' compactification.
\end{rema}
It is well known that (see \cite{graham1985conformal}) if $(M,g^+)$ is a $4-$dimensional conformally compact Einstein metric with boundary metric $h$ and  $g=r^2g^+$ is the geodesic compactification associated
with $h,$ then according to the Gauss lemma, $g_+=r^{-2}(dr^2+g_r).$
$$g_r=h+g^{(2)}r^2+g^{(3)}r^3+\cdots$$
where $g^{(2)}$ is the Schouten tensor and is determined by $h.$ $g^{(3)}$ is determined by $g^+$ and $h$ and hence it is a non-local term. The rest of power series is determined by $g^{(3)}$ and $h.$ This property is also true for higher dimension. From this point of view, Helliwell's condition of $C^{3,\alpha}$ initial compactification seems very natural. That we improve it to  $C^{2,\sigma}$ is a big step as we don't need any information about the non-local term.
\par The outline of this paper is as follows. In section \ref{sect2}, we introduce some basic facts about conformally compact Einstein metrics. We show that the Yamabe compactification  exists. The conditions in theorem \ref{theorem 1.1} are unchanged under this compactification. We also consider the Bach equation in dimension 4 and it is an elliptic PDE of second order about Ricci tensor if the scalar curvature is constant. At last, we introduce the harmonic coordinates.
\par In section \ref{sect3}, we deduce some boundary conditions. Including the Dirichlet condition of metric and Ricci curvature, the  Neumann condition of Ricci curvature and the oblique derivative condition of metric. We prove that these conditions are true even if the compactification $g$ is only $C^2.$
\par In section \ref{sect4}, we attempt to prove the main theorem. The first difficulty is $C^\alpha$ and $C^{1,\alpha}$ estimate of Ricci curvature. So we present the intermediate Schauder theory to solve the problem. Then we finish our proof  with the classical Schauder theory. In the end, with the help of Bach equation, we prove the regularity of defining function in the new coordinates.
\section{Preliminaries}  \label{sect2}
Let $(M,g_+)$ be a $n+1-$dimensional conformally compact Einstein manifold and $g=\rho^2g^+$ is a compactification. Then
\begin{equation}\label{2.1}
K_{ab}=\frac{K_{+ab}+|\nabla\rho|^2}{\rho^2}-\frac{1}{\rho}[D^2\rho(e_a,e_a)+D^2\rho(e_b,e_b)],
\end{equation}
\begin{equation}\label{2.2}
Ric=-(n-1)\frac{D^2\rho}{\rho}+[\frac{n(|\nabla\rho|^2-1)}{\rho^2}-\frac{\Delta\rho}{\rho}]g,
\end{equation}
\begin{equation}\label{2.3}
S=-2n\frac{\Delta\rho}{\rho}+n(n+1)\frac{|\nabla\rho|^2-1}{\rho^2}.
\end{equation}
Here $K_{ab},Ric,S$ are the sectional curvature, Ricci curvature and scalar curvature of $g$ and $D^2$ denote the Hessian. (Readers can see \cite{besse2007einstein} for the conformal transformation law of curvatures.)
\par If $g$ is a $C^2$ compactification, then from (\ref{2.3}),$|\nabla\rho|=1$ on $\M.$  Then by (\ref{2.1}) $K_{+ab}$ tends to $-1$ as $\rho\rightarrow 0.$ Hence a $C^2$ conformally compact Einstein manifold is asymptotically hyperbolic.
Let $D^2\rho|_{\M}=A$ denote the second fundamental form of $\M$ in $(\overline{M},g).$
The equation (\ref{2.2}) further implies that $\M$ is umbilic.
\subsection{Constant Scalar Curvature Compactification}
\begin{lemm}
Let $(M,g_+)$ be a conformally compact n-manifold with a $W^{2,p}$ conformal compactification $g=\rho^2g_+$ where $p> n/2.$ Suppose that $h=g|_{\M}$ is the boundary metric. Then there exits a $W^{2,p}$ constant scalar curvature compactification $\hat{g}=\hat{\rho}^2g_+$ with boundary metric $h.$
\end{lemm}
\begin{proof}
We only need to solve a Yamabe problem with Dirichlet data. Let $\hat{g}=u^{\frac{4}{n-2}}g$, then we consider the equation
 \begin{equation}\label{2.4}
 \left\{
    \begin{array}{l}
    \Delta_gu-\frac{n-2}{4(n-1)}Su+\frac{n-2}{4(n-1)}\lambda u^{\frac{n+2}{n-2}}=0
    \\ u>0\ \ in\ \overline{M}
    \\ u\equiv 1\ \ on\  \M
    \end{array}
 \right.
 \end{equation}
 In \cite{li1995yamabe}, Ma Li proved that the equation has a $C^{2,\alpha}$ solution if the metric $g$ is $C^{2,\alpha}$ when $\lambda=-1.$ Now we extend his conclusion in the case that $g\in W^{2,p}$ for some $p> n/2.$
 Let $\lambda=-1,$ and we consider the following functional
  $$I(u)=\frac{1}{2}\int_M(|\nabla u|^2+\frac{n-2}{4(n-1)}Su^2)dv+\frac{n-2}{2n}\int_M\frac{n-2}{4(n-1)}|u|^{\frac{2n}{n-2}}dv$$on
  the set
  $$A=\{u\in H^1(M):u|_{\M}=1\}.$$
  It is coercive and weakly lower semi-continuous. Then $I$ attains its infimum in $A,$ which means that $\exists u\in A,\ I(u)=\inf\limits_{v\in A} I(v).$  Since for any $\eta\in H^1_0(M), t\in\mathbb{R}, u+t\eta\in A,$ we have that
   $$\frac{d}{dt}I(u+t\eta)|_{t=0}=0.$$
   Then $u$ is a $H^1$ weak solution. By the Sobolev embedding theorem it follows that $u\in L^{\frac{2n}{n-2}}.$ Now let $$f(x)=\frac{n-2}{4(n-1)}(S+ u^{\frac{4}{n-2}}),$$ then $u$ is a weak solution of
   $-\Delta_gu+fu=0,\ u|_{\M}=1$ and $f\in L^{\frac{n}{2}}.$ By a standard method in PDE we can infer that $u\in L^q$ for any $q\geq 2.$ (One can see more details in \cite{brezis1978remarks}, theorem 2.3.) So $fu\in L^{p'}$ for any $p'<p$ and it implies that $u\in W^{2,p'}.$ If choose $p'>n/2,$ then $u$ is H\"older continuous. Then $f\in L^p,$ and finally we get that  $u\in W^{2,p}.$ The strong maximum principle tells
  us that $u$ is positive in M.
\end{proof}
\par If $g\in C^2$ and $S_g\in C^{\sigma}$ for some $\sigma>0,$ we know that the equation \ref{2.4} has a $C^{2,\sigma}$ solution $u.$ Then $\hat{g}=u^{\frac{4}{n-2}}g$ is still $C^2$, and the new defining function $\hat{\rho}=u\rho\in C^{2,\sigma}$.
In the following of this section, we don't distinguish $g$ with $\hat{g}.$ When we refer to the compactification $g,$ we mean the scalar curvature of $g$ is $-1$ near the boundary and the defining function is $C^{2,\sigma}.$
\subsection{The Bach Equation}
For a $4-$dimensional manifold, the Bach tensor is a conformal invariant and vanishes for Einstein metric, see \cite{besse2007einstein}. In local coordinates,
\begin{equation}\label{2.5}
B_{ij}=P_{ij,k}^{\ \ \ \ k}-P_{ik,j}^{\ \ \ \ k}-P^{kl}W_{kijl}
\end{equation}
where  $P_{ij}=\frac{1}{2}R_{ij}-\frac{S}{12}g_{ij}$ is the Schouten tensor.
\par Let $\{y^\beta\}_{\beta=0}^3$ be the smooth structure on $\overline{M}$ and when restricted on $\M,$ $\{y^i\}_{i=1}^3$ is smooth
structure of $\M.$ From above we can assume that $g\in C^\infty(M)\cap C^2(\overline{M}), S_g\equiv -1.$ Then the fact that $g_+$ is
Einstein and (\ref{2.5}) imply that
\begin{equation}\label{2.6}
\Delta Ric_{\alpha\beta}=\Gamma\ast\partial Ric+\mathcal{Q}
\end{equation}
in y-coordinates. Here $\Delta=g^{\alpha\beta}\partial_\alpha\partial_\beta,$ $\Gamma$ is the Christoffel symbol of $g,$ $\Gamma\ast\partial Ric$ denote the bilinear form of $\Gamma$ and $\partial Ric$ and $\mathcal{Q}$ denotes a quadratic curvature term.
\subsection{The Harmonic Coordinates Near Boundary}
In the rest of the paper, if there are no special instructions, any use of indices will
follow the convention that Roman indices will range from 1 to n, while Greek indices range from 0 to n.
\par We call the coordinates $\{x^\beta\}_{\beta=0}^n$ harmonic coordinates with respect to $g$ if $$\Delta_gx^\beta=0$$
 for $0\leq\beta\leq n.$ We are now going to construct harmonic coordinates in a neighbourhood of $\M$   if $g$ is smooth.
\par In fact, if $g\in C^{1,\alpha},\alpha\in(0,1)$ for any point $p\in\M,$ there is a neighbourhood $V$ and smooth structure $\{y^\beta\}_{\beta=0}^n$ where $y^0|_{\M}=0.$ Then by solving a local Dirichlet problem:
\begin{equation}   \label{2.7}
\left\{\begin{array}
     {l}
    \Delta_g x^\beta=0 \ in \ V
    \\x^\beta|_{V\cap\M}=y^\beta|_{V\cap\M},
    \end{array}\right.
\end{equation}
there is a $C^{2,\alpha}$ solution by \cite{gilbarg2015elliptic} and we have the Schauder estimate:
$$\parallel x^\beta-y^\beta\parallel_{C^{2,\alpha}(V)}\leq C(\parallel \Delta(x^\beta-y^\beta)\parallel)_{C^\alpha(V)}+\parallel x^\beta-y^\beta\parallel_{C^{2,\alpha}(\partial V)}=C\parallel\Delta y \parallel_{C^\alpha(V)} $$
We can assume that the y-coordinates is the normal coordinates at $p,$ then $\Delta y(p)=0.$ Hence if $V$ is small enough,$\parallel x^\beta-y^\beta\parallel_{C^{2,\alpha}(V)}$ tends to $0.$ $\{x^\beta\}_{\beta=0}^n,0\leq\beta\leq n$ is a coordinate around $p.$
\par In particular, if $g\in C^2,$ then the solution $x\in C^{2,\alpha}(y)$ for any $\alpha\in(0,1).$ Hence
$$g_{\alpha\beta}=g(\frac{\partial}{\partial x^\alpha},\frac{\partial}{\partial x^\beta})\in C^{1,\alpha}(\overline{M})$$
In harmonic coordinates $\{x^\beta\}_{\beta=0}^n$, the Ricci tensor could be written as:
$$\Delta g_{ij}=-2R_{ij}+Q(g,\partial g)$$
where $Q(g,\partial g)$ is a polynomial of $g$ and $\partial g.$ For more details, one can see \cite{deturck1981some}.
\par Here we refer to the special coordinates constructed in section 4 in \cite{gicquaud2013conformal}. Instead of the harmonic coordinates above, those coordinates may also be useful in our situation, and may also help us to deal with it in higher dimension of even number. That's an interesting problem.
\section{The Boundary Conditions}\label{sect3}
\par In this section, we derive a boundary problem for $g$ and Ricci curvature of a conformal compact Einstein manifold in the harmonic coordinates as defined in section 2. We do it locally, that is, for any $p\in \M,$ there is a neighborhood $V$ contains $p$ and a local harmonic chart $\{x^\beta\}.$  Let $D=V\cap \M$ be the boundary portion and let $g\in C^2(V)$ be the Yamabe compactification. We will give the Dirichlet and Neumann boundary conditions of $g$ and $Ric(g)$ on D. Here we state that the boundary conditions in this section hold for all dimension.
\par In fact, as it is showed in \cite{helliwell2008boundary} and \cite{jin2019boundary} that, if $g$ is $C^{3,\alpha}$ compact, we have following boundary conditions:
\begin{prop}
Let $(M,g^+)$ be a $n+1$-dimensional conformally compact Einstein manifold with a $C^{3,\alpha}$ Yamabe compactification $g=\rho^2g^+.$ $g|_{\M}=h$ is the boundary metric. Suppose that $\{x^\beta\}_{\beta=0}^n$ are any coordinates near the boundary such that $x_0$ is defining function and $\{x^i\}_{i=0}^n$
 are coordinates of $\M.$ We have:
\begin{equation}\label{3.1}
g_{ij}=h_{ij}.
\end{equation}
\begin{equation}\label{3.2}
R_{ij}=\frac{n-1}{n-2}(Ric_h)_{ij}+(\frac{1}{2n}S-\frac{1}{2(n-2)}S_h)h_{ij}+\frac{n-1}{2n^2}H^2h_{ij}.
\end{equation}
\begin{equation}\label{3.3}
R_{0i}=-(g^{00})^{-\frac{1}{2}}\frac{n-1}{n}\frac{\partial H}{\partial x_i}-\frac{g^{0j}}{g^{00}}R_{ij}.
\end{equation}
\begin{equation}\label{3.4}
R_{00}=\frac{1}{(g^{00})^2}(g^{0i}g^{0j}R_{ij}+g^{00}(\frac{1}{2}(S-S_h)-\frac{n-1}{2n}H^2)).
\end{equation}
\begin{equation}\label{3.5}
N(R_{0i})=(g^{00})^{-\frac{1}{2}}(-g^{j\beta}\partial_\beta R_{ji}+g^{\eta\beta}\Gamma_{i\beta}^\tau R_{\eta\tau})
\end{equation}
where $N=\frac{\nabla x_0}{|\nabla x_0|}=(g^{00})^{-\frac{1}{2}}g^{0\beta}\partial_\beta$ is the unit norm vector on $\M$ and $R_{\alpha\beta},S,H$ are Ricci curvature, scalar curvature mean curvature respect to $g.$
\end{prop}
The formula (3.1) is trivial and (3.5) is deduced by the second Bianchi identity and the fact that the scalar curvature is constant near the boundary. Here we briefly recall the proof of The formula (3.2), (3.3) and (3.4). For a $C^{3,\alpha}$  conformally compact Einstein metric, there is a unique $C^{2,\alpha}$  geodesic compactification with the same boundary metric (lemma 5.1 in \cite{lee1994spectrum}). Then for such a $C^2$ geodesic compactification, we have a good formula for Ricci curvature and scalar curvature on the boundary. At last, we use the Ricci formula under conformal change to get  (3.2), (3.3) and (3.4).
\par In this section, we will show that the formula (3.2), (3.3) and (3.4) still hold for $C^2$  conformally compact Einstein metric.

\par In fact, if $g$ is $C^2$ conformally compact, then there exists a sequence of $C^{3,\alpha}(\overline{M})$ metrics $g_k$ which converge to $g$ in $C^2$ norm in smooth structure of $\overline{M}.$ However $g_k$ are not conformal Einstein in general. In the following, we omit the index $k$ and  assume that $g$ is a $C^{3,\alpha}$ metric on $\overline{M}.$ By choosing a defining function $\rho$ satisfying $|\nabla\rho|_g=1$ on $\M,$ we make $g^+=\rho^{-2}g.$ Then with Taylor theorem, there is a $C^{2,\alpha}$ function $b$ such that $|\nabla\rho| ^2=1+b\rho$ near the boundary.
\begin{equation}\label{3.6}
Ric=-(n-1)\frac{D^2\rho}{\rho}+[\frac{n(|\nabla\rho|^2-1)}{\rho^2}-\frac{\Delta\rho}{\rho}]g+\frac{F}{\rho},
\end{equation}
\begin{equation}\label{3.7}
S=-2n\frac{\Delta\rho}{\rho}+n(n+1)\frac{|\nabla\rho|^2-1}{\rho^2}+\frac{trF}{\rho},
\end{equation}
where $F=\rho(Ric_{g_+}+ng_+)=\rho Ric_g+(n-1)D^2\rho-(nb-\Delta\rho)g\in C^{1,\alpha}(\overline{M}).$
\\ Now we prove the following formulas:
$$R_{0i}=-(g^{00})^{-\frac{1}{2}}\frac{n-1}{n}\frac{\partial H}{\partial x_i}-\frac{g^{0j}}{g^{00}}R_{ij}+Q(F,DF,h,Dg,H),$$
$$R_{00}=\frac{1}{(g^{00})^2}(g^{0i}g^{0j}R_{ij}+g^{00}(\frac{1}{2}(S-S_h)-\frac{n-1}{2n}H^2))+Q(F,DF,h,Dg,H),$$
\begin{equation}\label{3.8}
R_{ij=}\frac{n-1}{n-2}(Ric_h)_{ij}+(\frac{1}{2n}S-\frac{1}{2(n-2)}S_h)h_{ij}+\frac{n-1}{2n^2}H^2h_{ij}+Q(F,DF,h,Dg,H).
\end{equation}
Here $h=g|_{\M},$ $H$
is the mean curvature, $Q$ is a polynomial and $Q(F,DF,h,Dg,H)=0$ if $F=DF=0$ on $\M.$
We will use three lemmas to prove (\ref{3.8}).
\par First, there is a unique $C^{2,\alpha}$ geodesic compactification of $g^+$ with boundary metric $h$ and denote it by $\g=r^2g^+.$ Let $\bar{g}=u^2g$ where $u=\frac{r}{\rho}$ satisfying that $\equiv 1$ on the boundary and $u\in C^{2,\alpha}.$
Then $\F=r(Ric_{g_+}+ng_+)=uF$ is still $C^{1,\alpha}(\overline{M}).$ We will calculate the boundary curvature  of $\g$ and notice that the second fundamental form of $\g$ at $\M$ is not 0, but determined by the tensor $\F.$
\begin{lemm}\label{lemma 3.2}
Suppose that $\g=r^2g_+$ is a $C^2$ conformally compactification of manifold $(M,g_+)$ with boundary metric $h.$ Then on the boundary $\M,$
\begin{equation}\label{3.9}
\S=\frac{n}{n-1}(S_h)+Q(\F,D\F),
\end{equation}
\begin{equation}\label{3.10}
\R_{ij}=\frac{n-1}{n-2}(Ric_h)_{ij}-\frac{1}{2(n-1)(n-2)}S_hh_{ij}+Q(\F,D\F,h,D\g).
\end{equation}
Here $\S$ and $\R_{ij}$ are  the scalar curvature and Ricci curvature of $\g.$ $Q$ is a polynomial satisfying  $Q(\F,D\F,h,D\g)=0$ if $\F=D\F=0.$
\end{lemm}
\begin{proof}
Let us choose the coordinates $(r,y^1,\cdots,y^n),$ near $\M$ such that $\g=dr^2+g_r,$ i.e.
$$g_{ri}=g^{ri}=0, g_{rr}=g^{rr}=1.$$
According to Gauss Codazzi equation,
\begin{equation}\label{3.11}
\begin{aligned}
\R_{ij}&=\g^{\alpha\beta}\R_{i\alpha\beta j}\\
       &=\g^{kl}((R_h)_{iklj}+\A_{il}\A_{kj}-\A_{ij}\A_{kl})+\R_{irrj}\\
       &=(R_h)_{ij}+\g^{kl}\A_{il}\A_{kj}+\H\A_{ij}+\R_{irrj}.
\end{aligned}
\end{equation}
Taking trace with respect to $i$ and $j$,
\begin{equation}\label{3.12}
\R_{rr}=\frac{1}{2}(\S-S_h+\H^2-\g^{ij}\g^{kl}\A_{il}\A_{kj}).
\end{equation}
Then
\begin{equation}\label{3.13}
\begin{aligned}
\R_{irrj}&=\g(\bar{\nabla}_{\partial_i}\bar{\nabla}_{\partial_r}\partial_r,\partial_j)-\g(\bar{\nabla}_{\partial_r}\bar{\nabla}_{\partial_i}\partial_r,\partial_j)-
              \g(\bar{\nabla}_{[\partial_r,\partial_i]}\partial_r,\partial_j)\\
         &=-\partial_r\g(\bar{\nabla}_{\partial_i}\partial_r,\partial_j)+\g(\bar{\nabla}_{\partial_i}\partial_r,\bar{\nabla}_{\partial_r}\partial_j)\\
         &=-\partial_r\A_{ij}+\A^2(\partial_i,\partial_j).
\end{aligned}
\end{equation}
From (\ref{3.6}) and (\ref{3.7}) ,we have:
\begin{equation}\label{3.14}
\begin{aligned}
    &\R_{ij}=-(n-1)\frac{\A_{ij}}{r}-\frac{\bar{\Delta}r}{r}\g_{ij}+\frac{\F_{ij}}{r},
    \\
    &\R_{ri}=\frac{\F_{ri}}{r},
    \\
    &\R_{rr}=-\frac{\bar{\Delta}r}{r}+\frac{\F_{rr}}{r},
    \\
    &\S=-2n\frac{\bar{\Delta}r}{r}+\frac{tr\F}{r}.
    \end{aligned}
\end{equation}
$\R$ic and $\S$ is continuous on $\overline{M},$ so on $\M \ (r=0)$ we have:
\begin{equation}\label{3.15}
\begin{aligned}
    &\A_{ij}=\frac{1}{n-1}(\F_{rr}h_{ij}-\F_{ij}),\\
    &\H=\bar{\Delta}r=\F_{rr}=\frac{1}{2n}tr\F.
\end{aligned}
\end{equation}
Hence
\begin{equation}\label{3.16}
\begin{aligned}
    &\R_{ij}=-(n-1)\partial_r\A_{ij}-\partial_r\bar{\Delta}r\g_{ij}-\bar{\Delta}r\partial_r\g_{ij}+\partial_r\F_{ij},\\
    &\R_{ri}=\partial_r\F_{ri},\\
    &\R_{rr}=-\partial_r\bar{\Delta}r+\partial_r\F_{rr},\\
    &\S=-2n\partial_r\bar{\Delta}r+\partial_rtr\F.
\end{aligned}
\end{equation}
Combining all the formulas above, we get that
\begin{equation}\label{3.17}
\begin{aligned}
\S&=\frac{n}{n-1}(S_h-\H^2+|\A|^2_h-\frac{1}{n}\partial_rtr\F)\\
  &=\frac{n}{n-1}(S_h-\F_{rr}^2+\frac{1}{(n-1)^2}(n\F_{rr}^2+\F_{rr}tr_h\F+|\F|_h^2)-\frac{1}{n}\partial_rtr\F)
\end{aligned}
\end{equation}
which is (\ref{3.9}).
\begin{equation}\label{3.18}
\begin{aligned}
\R_{ij}&=\frac{n-1}{n-2}((R_h)_{ij}+\g^{kl}\A_{il}\A_{kj}+\H\A_{ij})-\F_{rr}\partial_r\g_{ij}+\partial_r\F_{ij}\\
  &+\frac{1}{n-2}(\A^2_{ij}-\frac{1}{2}(\S-S_h+\H^2-|A|^2)+\partial_r\F_{rr})h_{ij}
\end{aligned}
\end{equation}
Noticing that $\A_{ij}$ is totally determined by $\F$ and $h,$ hence (\ref{3.10}) holds.
\end{proof}

\begin{lemm}\label{lemma 3.3}
Let $g=\rho^2g_+$ be a $C^{3,\alpha}$ conformally compact metric of $(M,g_+)$ and $\g=r^2g_+$ be a $C^{2,\alpha}$ geodesic compactification with the same boundary metric $g|_{\M}=\g|_{\M}=h.$ Let $r=u\rho, A=D^2\rho,$ then  $A|_{\M}=\A-u_rh.$
\end{lemm}
\begin{proof}
In the local coordinates $(r,y^1,y^2,\ldots,y^n)$ near $\M,$ $\A_{ij}=-\bar{\Gamma}_{ij}^r.$ Then the relationship between the connection $\nabla$ of $g$ and $\bar{\nabla}$ of $\g$ is:
$$\Gamma_{ij}^r=\bar{\Gamma}_{ij}^r-\frac{1}{u}(\delta^r_ju_i+\delta^r_iu_j-g_{ij}u_r)=\frac{1}{u}u_rh_{ij}.$$
$g=u^{-2}\g, grad_g=u^2 grad_{\g},$ then
\begin{equation}\label{3.19}
\begin{aligned}
    A_{ij}&=D^2\rho(\partial_i,\partial_j)=g(\nabla_{\partial_i}\nabla\rho,\partial_j)=-g(\nabla\rho,\nabla_{\partial_i}\partial_j)\\
          &=-\Gamma_{ij}^rg(\nabla\rho,\partial_r)=-\Gamma_{ij}^r\g(\bar{\nabla}\rho,\partial_r)\\
          &=-\Gamma_{ij}^r\g(\bar{\nabla}(\frac{r}{u}),\partial_r)=-\Gamma_{ij}^r\g(\frac{u\bar{\nabla}r-r\bar{\nabla}u}{u^2},\partial_r)\\
          &=-\Gamma_{ij}^r\g(\bar{\nabla}r,\bar{\nabla}r)=\A_{ij}-u_rh_{ij}
\end{aligned}
\end{equation}
\end{proof}
Lemma 3.3 tells us that $u_r=\frac{\H-H}{n}.$ Using the fact that $u|_{\M}\equiv 1,$
$$\bar{\nabla}u=\frac{\H-H}{n}\bar{\nabla}r.$$
\begin{lemm}\label{lemma 3.4}
Suppose that $g,\g$ are defined as in lemma \ref{lemma 3.3}, then on the boundary $\M,$
$$R_{ri}=\frac{n-1}{n}\frac{\partial(\H-H)}{\partial x_i}+Q(\F,D\F,H),$$
$$R_{rr}=\frac{1}{2}(S-S_h)-\frac{n-1}{2n}H^2+Q(F,DF,H),$$
\begin{equation}\label{3.20}
R_{ij}=\R_{ij}+(\frac{1}{2n}(S-\S))h_{ij}+\frac{n-1}{2n^2}H^2h_{ij}+Q(\F,D\F,H).
\end{equation}
Here $Q(\F,D\F,H)=0$ if $F=DF=0.$
\end{lemm}
\begin{proof}
Let $g=u^{-2}\g,$ then
$$Ric=\R ic+(n-1)\frac{\bar{D}^2u}{u}+(\frac{\bar{\Delta}u}{u}+\frac{n|\bar{\nabla}u|^2_{\g}}{u^2})\g.$$
We also know that
$$\bar{\Delta}u=div\bar{\nabla}u=div(\frac{\H-H}{n}\bar{\nabla}r)=\frac{\H-H}{n}\bar{\Delta}r+\frac{\partial_r(\H-H)}{n},$$
$$\D^2u(\partial_i,\partial_j)=\g(\bar{\nabla}_{\partial_i}\bar{\nabla}u,\partial_j)=\frac{\H-H}{n}\A_{ij},$$
$$\D^2u(\partial_i,\partial_r)=u_{ir}=\frac{1}{n}\frac{\partial(\H-H)}{\partial x_i},$$
$$\D^2u(\partial_r,\partial_r)=\frac{\partial_r(\H-H)}{n}=\bar{\Delta}u-\frac{\H-H}{n}\bar{\Delta}r.$$
 Then on $\M,$ we conclude that
$$R_{ri}=\R_{ri}+\frac{n-1}{n}\frac{\partial(\H-H)}{\partial x_i}=\partial_r\F_{ri}+\frac{n-1}{n}\frac{\partial(\H-H)}{\partial x_i},$$
$$R_{rr}=\R_{rr}+n\bar{\Delta}u-\frac{(n-1)(\H-H)}{n}\bar{\Delta}r+\frac{(\H-H)^2}{n},$$
\begin{equation}\label{3.21}
R_{ij}=\R_{ij}+(\bar{\Delta}u+\frac{(\H-H)^2}{n})h_{ij}+(n-1)\frac{\H-H}{n}\A_{ij}.
\end{equation}
Taking trace with respect to $i$ and $j,$
$$S=\S+2n\bar{\Delta}u+\frac{n+1}{n}(\H-H)^2-\frac{(n-1)(\H-H)}{n}\bar{\Delta}r+\frac{n-1}{n}(\H-H)\H.$$
Then
\begin{equation}\label{3.22}
    \bar{\Delta}u=\frac{1}{2n}(S-\S-\frac{n+1}{n}(\H-H)^2+\frac{(n-1)(\H-H)}{n}\bar{\Delta}r-\frac{n-1}{n}(\H-H)\H).
\end{equation}
The result follows from (\ref{3.21}) and (\ref{3.22}).
\end{proof}
In the end, lemma \ref{lemma 3.2} and lemma \ref{lemma 3.4} imply (\ref{3.8}).
\par With the preparation above, let's consider a $C^2$ conformally compact Einstein metric $g=\rho^2g^+$ on $(\overline{M},y).$ We can choose a sequence of $C^{3,\alpha}$ metric $g_k$ which converge to $g$ in $C^2(\overline{M})$ norm. Let $\rho_k=\frac{\rho} {|\nabla^{g_k}\rho|_{g_k}},$ so $\rho_k\in C^{3,\alpha}(\overline{M})$ and $|\nabla^{g_k}\rho_k|_{g_k}\equiv 1$ on $\M.$ Let $g^+_k=(\rho_k)^{-2}g_k,$ then $g_k$ is a $C^{3,\alpha}$ conformally compactification of $(M,g^+_k)$ with defining function $\rho_k.$ Defining $F_k=\rho_k(Ric_{g^+_k}+ng^+_k)$ as above, then the formula of $Ric_{g_k}$ on $\M$ is like (\ref{3.10}) and $F_k$ converge to 0 in $C^1(\M)$ norm.
\\ Finally, as the Ricci curvature of $g_k$ converges to that of $g$ uniformly, we conclude that (\ref{3.7}), (\ref{3.19}) and (\ref{3.20}) hold.
\subsection{Other Boundary Conditions}
 \par We see that if the metric $g$ in lemma \ref{lemma 3.3} is conformally Einstein, then $\A=0$ on $\M$ and the boundary is umbilic. This conclusion is also true even if $g$ is $C^2$ compact and in this case the geodesic compactification is at least $C^1.$ (See \cite{lee1994spectrum}.) Then we have
$$A_{ij}=\frac{H}{n}h_{ij}.$$
Taking the derivative of the equation above along $\M,$
$$\partial_kA_{ij}=\frac{\partial_kH}{n}h_{ij}+\frac{H}{n}\partial_kh_{ij}.$$
Combining it with (\ref{3.3}), we get that
\begin{equation}\label{3.23}
\partial_kA_{ij}=-\frac{1}{n-1}(g^{00})^{\frac{1}{2}}(R_{0k}+\frac{g^{0j}}{g^{00}}R_{ij})
\end{equation}
Technically, this is not a boundary condition because both sides are of the second derivative of $g.$ However, this plays an important role in proving the regularity and we will use the condition later.
\\ \par If we choose harmonic coordinates, we also have the following boundary condition:
\begin{equation}\label{3.24}
g^{\eta\beta}\partial_\eta (g_{\alpha\beta}-\frac{1}{2}\partial_\alpha g_{\eta\beta})=0
\end{equation}
This is just the local expression of $\Delta_gx^\alpha=0.$
\section{Proof of the Main Theorem}\label{sect4}
We prove the main theorem in this section with the Bach equation in harmonic coordinates and some boundary conditions in last section. Firstly, let's recall some intermediate Schauder theory of elliptic PDE in \cite{gilbarg1980intermediate},\cite{lieberman1986intermediate}, i.e. $C^\alpha$ and $C^{1,\alpha}$ estimate.
\subsection{Intermediate Schauder Estimate}

 Suppse $\Omega$ is a bounded convex domain in $\mathds{R}^n$ and $a$ is a positive number satisfying  $a=k+\beta$ ($k\in\mathds{N}, \beta\in(0,1]$)
Defining
$$|u|_a=\sum\limits_{|\alpha|\leq k}|D^\alpha u|_0+\sum\limits_{|\alpha|=k}\sup\limits_{x,y\in \Omega}\frac{|D^\alpha u(x)-D^\alpha u(u)|}{|x-y|^\beta}.$$
Let $H_a(\Omega)$ denote the H\"older space of functions with finite norm $|u|_a$ on $\Omega,$ i.e. $H_a(\Omega)=C^ {k,\beta}(\overline{\Omega})$. Setting
$$\Omega_\delta=\{x\in\Omega|dist(x,\partial\Omega)>\delta\}$$
Let $b$ be a number satisfying $a+b\geq 0$ and define
$$|u|_{a,\Omega}^{(b)}=\sup\limits_{\delta>0}\delta^{a+b}|u|_{a,\Omega_\delta}$$
Let $H_a^{(b)}(\Omega)$ denote the space of functions $u$ in $H_a(\Omega_\delta),(\forall \delta>0)$ such that $|u|_{a,\Omega}^{(b)}$ is finite.
Let $H_a^{(b-0)}(\Omega)$ be the space of functions $u$ in $H_a^{(b)}(\Omega)$ such that if $\delta\rightarrow 0,$ then $\delta^{a+b}|u|_{a,\Omega_\delta}\rightarrow 0$.
\par Basic properties: (the following constant $C$ depends on $a,b,\Omega.$)
\begin{itemize}
\item[1.] $H_a^{(-a)}(\Omega)=H_a(\Omega)=C^{k,\beta}(\overline{\Omega})$. Noticing that if $a$ is positive integer, $H_a(\Omega)=C^{a-1,1}(\overline{\Omega});$
\item[2.] If $b\geq b'$, then $|u|_{a,\Omega}^{(b)}\leq C|u|_{a,\Omega}^{(b')};$
\item[3.] If $0\leq a'\leq a,a'+b\geq 0$ and $b$ is not a non-positive integer, then $|u|_{a',\Omega}^{(b)}\leq C|u|_{a,\Omega}^{(b)}$;
\item[4.] If $0\leq c_j\leq a+b, a\geq 0,j=1,2,$ then
$$|uv|_{a}^{(b)}\leq C(|u|_{a}^{(b-c_1)}|v|_{0}^{(c_1)}+|u|_{0}^{(c_2)}|v|_{a}^{(b-c_2)})$$
Specially, if $u$ and $v$ are continuous functions (bounded), then $|uv|_{a}^{(b)}\leq C(|u|_{a}^{(b)}+|v|_{a}^{(b)})$.
\end{itemize}
With the preparation above, we could state the intermediate Schauder theory. Assuming that $\Omega$ is a bounded $C^\gamma$ domain where $\gamma\geq 1$ and $a,b$ are not integer satisfying
$$0<b\leq a, \ \ a>2, \ \ b\leq\gamma$$
Let
$$P=\sum\limits_{|\alpha|\leq 2}p_\alpha(x)D^\alpha$$
be the elliptic differential operator of second order on $\overline{\Omega}$  where
$$p_\alpha\in H_{a-2}^{(2-b)}(\Omega), \ \ if \ |\alpha|\leq 2$$
$$p_\alpha\in H_0(\Omega),\ \ if \ |\alpha|= 2$$
$$p_\alpha\in H_{a-2}^{(2-|\alpha|-0)}(\Omega), \ \ if \ b<|\alpha|.$$
Then we have:
\begin{lemm}\label{lemma 4.1}[Theorem 6.1 in \cite{gilbarg1980intermediate}]
Let $P,a,b$ be defined as above. If $p_0\leq 0$ and the principal part of $P$ is positive, then the
Dirichlet problem
$$Pu=f \ \ in \ \Omega, \ \ \ u=u_0 \ \ on \ \partial\Omega$$
has a unique solution $u\in H_{a}^{(-b)}(\Omega)$  for every $f\in H_{a-2}^{(2-b)}(\Omega)$ and $u_0\in H_b(\partial\Omega,$ and we
have
$$u_{a}^{(-b)}(\Omega)\leq C(|u|_{b,\partial\Omega}+|f|_{a-2}^{(2-b)}(\Omega))$$
\end{lemm}
We also have the following regularity result:
\begin{lemm}\label{lemma 4.2}[Theorem 6.3 in \cite{gilbarg1980intermediate}]
Let $\Omega,P,a,b$ satisfy the hypotheses in lemma 4.1, and let $u\in C^0(\overline{\Omega})\cap C^2(\Omega),u|_{\partial\Omega}\in H_b(\partial\Omega),Pu\in  H_{a-2}^{(2-b)}(\Omega).$ Then it follows that $u\in H_{a}^{(-b)}(\Omega).$
\end{lemm}
For the boundary oblique derivative conditions, we have the following lemma:
\begin{lemm}\label{lemma 4.3}[Theorem 3 in \cite{lieberman1986intermediate}]
Let $a,b$ be non-integer and $1<b\leq a,a>2$ and $\Omega\subset\mathds{R}^n$ be bounded domain with $H_b$ boundary. Let
$$P=\sum\limits_{|\alpha|\leq 2}p_\alpha(x)D^\alpha \ \ in \ \Omega, \ \ \ M=\sum\limits_{|\alpha|\leq 1}m_\alpha(x)D^\alpha \ \ on \ \partial\Omega.$$
Here
$$\sum\limits_{|\alpha|=2}p_\alpha\xi^\alpha\geq c|\xi|^2 \ \forall \xi\in\mathds{R}^n,\ \ \ \sum\limits_{|\alpha|=1}m_\alpha v^\alpha>0$$
 where $c$ is a positive number. We also let
$$p_\alpha\in H_{a-2}^{(2-b)}(\Omega), \ \ if \ \ |\alpha|\leq 2;\ \ \ m_\alpha\in H_{b-1}(\partial\Omega) \ \ if \ \ |\alpha|\leq 1,$$
$$p_\alpha\in H_{a-2}^{(0-0)} \ if \  |\alpha|=2 \ and \ b<2.$$
(a) If $p_0\leq 0, \ m_0<0,$ then the oblique derivative problem
\begin{equation}\label{4.1}
Pu=f \ \ in \ \Omega, \ \ \ Mu=g \ \ on \ \partial\Omega
\end{equation}
has a unique solution $u\in H_{a}^{(-b)}(\Omega)$ for every $f\in H_{a-2}^{(2-b)}(\Omega)$ and $g\in H_{b-1}(\partial\Omega).$ Moreover,
$$u_{a}^{(-b)}(\Omega)\leq C(|g|_{b-1,\partial\Omega}+|f|_{a-2}^{(2-b)}(\Omega)).$$
(b) If $u\in C^0(\overline{\Omega})\cap C^2(\Omega)$ is a solution of (\ref{4.1}) with $f\in H_{a-2}^{(2-b)}(\Omega)$,$g\in H_{b-1}(\partial\Omega)$ and the directional derivative
$\sum\limits_{|\alpha|=1}m_\alpha$ exists at each point of $\partial\Omega,$ then $u\in H_{a}^{(-b)}(\Omega).$
\end{lemm}

\subsection{The $C^{1,\sigma}$ Regularity of Ricci Curvature}
For a $C^2$ conformally compact Einstein metric $g=\rho^2g_+,$  $\rho\in C^{2,\sigma},$  we know that  $Ric\in C^0(\overline{M})$ in the initial smooth y-coordinates. We observe that from (\ref{2.2})
$$\rho Ric=-(n-1)D^2\rho+[\frac{n(|\nabla\rho|^2-1)}{\rho}-\Delta\rho]g=Q(\partial g,\partial^2\rho)\in C^\sigma(\overline{M},\{y\}).$$
Now we compute the metric and curvature in harmonic coordinates $\{x^\beta\}_{\beta=0}^3.$ As $g$ is $C^2,$ $x\in C^{2,\lambda}(y), \forall \lambda\in(0,1).$ Then in x-coordinates, we have that
\begin{equation}\label{4.2}
Ric(\frac{\partial}{\partial x^\alpha},\frac{\partial}{\partial x^\beta})=\frac{\partial y^\gamma}{\partial x^\alpha}\frac{\partial y^\tau}{\partial x^\alpha}Ric(\frac{\partial}{\partial y^\gamma},\frac{\partial}{\partial y^\tau})\in C^0(\overline{M},\{x\})
\end{equation}
\begin{equation}\label{4.3}
\rho Ric(\frac{\partial}{\partial x^\alpha},\frac{\partial}{\partial x^\beta})=\rho\frac{\partial y^\gamma}{\partial x^\alpha}\frac{\partial y^\tau}{\partial x^\alpha}Ric(\frac{\partial}{\partial y^\gamma},\frac{\partial}{\partial y^\tau})\in C^\sigma(\overline{M},\{x\})
\end{equation}
By lemma 4.4 below, we conclude that  $Ric\in H_\sigma^{(1-\sigma)}(\overline{M}).$
\begin{lemm}\label{lemma 4.4}
Suppose that $f$ is a continuous function on $\overline{M}$ and $\rho f\in C^\sigma(\overline{M}),$ then $f\in H_\sigma^{(1-\sigma)}(\overline{M}).$
\end{lemm}
\begin{proof}
As $|\nabla\rho|\equiv 1$ on $\M,$ we can assume that $\frac{1}{2}\leq|\nabla\rho|\leq 2$ on $\M\times[0,\epsilon)$ for a small $\epsilon>0.$
Let $$\Omega_\delta=\{x\in M|dist(x,\partial M>\delta\},\ \ \ M_\delta=\{x\in M|\rho(x)>\delta\}.$$
A direct calculation shows that $$\Omega_\delta\subset M_{\frac{\delta}{2}}\subset\Omega_{\frac{\delta}{4}}$$
So we don't distinguish $\Omega_\delta$ and $M_\delta$ when studying the definition of $|u|_{a,\Omega}^{(b)}.$
Since $\rho f\in C^\sigma(\overline{M}),$ for any $x,y\in M_\delta,$
$$\frac{|\rho(x)f(x)-\rho(y)f(y)|}{d^\sigma(x,y)}\leq C.$$
Then
$$C\geq\frac{|\rho(x)f(x)-\rho(x)f(y)+\rho(x)f(y)-\rho(y)f(y)|}{d^\sigma(x,y)}\geq\rho(x)\frac{|f(x)-f(y)|}{d^\sigma(x,y)}-|f(y)|\frac{|\rho(x)-\rho(y)|}{d^\sigma(x,y)},$$
which means
$$\rho(x)\frac{|f(x)-f(y)|}{d^\sigma(x,y)}\leq C+|f|_{0,M_\delta}|\rho|_\delta.$$
By assumption, $f$ is continuous, in particular, $f$ is bounded. As a consequence, $\delta|f|_{\sigma,M_\delta}<C'$ for any $\delta>0.$
This proves the lemma.
\end{proof}
\begin{lemm}\label{lemma 4.5}
In harmonic coordinates, $g\in H_{2+\sigma}^{(-1-\sigma)}(\overline{M}).$
\end{lemm}
\begin{proof}
In harmonic charts,
$$\Delta g_{\alpha\beta}=-2R_{\alpha\beta}+Q(g,\partial g)$$
Let $a=2+\sigma,\ b=1+\sigma,$ then according to lemma \ref{lemma 4.2}, $g_{\alpha\beta}\in H_{2+\sigma}^{(-1-\sigma)}(\overline{M}).$
\end{proof}
Now we have that $g\in H_{2+\sigma}^{(-1-\sigma)}(\overline{M}),$ so the curvature $Rm\in H_\sigma^{(1-\sigma)}(\overline{M}).$ By linear transformation of tensor in coordinate system (similar to (\ref{4.2})), $Rm$ is still continuous in x-coordinates. Recall that $\mathcal{Q}$ in (\ref{2.6}) is the quadratic term of curvature, then $\mathcal{Q}\in H_\sigma^{(1-\sigma)}(\overline{M})$ from the basic property 4 in section 4.1.
\par As $g\in C^2(\overline{M},y),$ \ref{3.2}),(\ref{3.3}) hold in y-coordinates on $\M.$ In the harmonic coordinates $\{x^\beta\}_{\beta=0}^3,$
$$\frac{\partial x^\alpha}{\partial y^i}|_{\M}=\frac{\partial x^\alpha|_{\M}}{\partial y^i}=\delta_i^\alpha.$$
Then on $\M,$
$$Ric(\frac{\partial}{\partial y^i},\frac{\partial}{\partial y^j})=\frac{\partial x^\gamma}{\partial y^i}\frac{\partial x^\tau}{\partial y^j}Ric(\frac{\partial}{\partial x^\gamma},\frac{\partial}{\partial x^\tau})=\delta_i^\gamma\delta_j^\tau Ric(\frac{\partial}{\partial x^\gamma},\frac{\partial}{\partial x^\tau})=Ric(\frac{\partial}{\partial x^i},\frac{\partial}{\partial x^j}),$$
 $$Ric(\frac{\partial}{\partial x^0},\frac{\partial}{\partial x^\alpha})=\frac{\partial y^\gamma}{\partial x^0}\frac{\partial y^\tau}{\partial x^\alpha}Ric(\frac{\partial}{\partial y^\gamma},\frac{\partial}{\partial y^\tau}).$$
For any $p\in \M,$ consider the $C^{2,\lambda}$ harmonic chart $(V,\{x^\theta\}_{\theta=0}^3)$ around $p.$ Let $D=V\cap\M$ be the boundary portion. Then the Bach equation (\ref{2.6}) could be written as
\begin{equation}\label{4.4}
\Delta Ric_{\alpha\beta}=\frac{\partial}{\partial y^\theta}f_{\alpha\beta}^\theta+\mathcal{Q}.
\end{equation}
Here $f_{\alpha\beta}^\theta=\Gamma\ast Ric\in H_{\sigma}^{(1-\sigma)}(\overline{M})$ , $\theta=0,1,2,3.$ We will firstly deal with the $R_{ij}$ term where $1\leq i,j\leq 3.$ Consider the following equations:
 \begin{equation}\label{4.5}
 \left\{\begin{array}
     {l}
         \Delta u_{ij}^0=f_{ij}^0 \ in \ V
    \\ \frac{\partial}{\partial y^0}u_{ij}^0=0 \ on \ D\\
    \Delta u_{ij}^k=f_{ij}^k \ in \ V
    \\ u_{ij}^k=0 \ \ \ \ \ \ on \ D
     \end{array}\right.
 \end{equation}
where $k=1,2,3.$ By lemma \ref{lemma 4.1} and lemma \ref{lemma 4.3}, the 4 equations above have solutions in $H_{2+\sigma}^{(-1-\sigma)}(V).$ Let $\tilde{R}_{ij}=R_{ij}-\partial_\theta u_{ij}^\theta,$ then
$$\Delta\tilde{R}ic_{ij}=\mathcal{Q}+Q(g,\partial g,\partial^2g,\partial u,\partial^2u)\in H_{\sigma}^{(1-\sigma)}(V)$$
From lemma \ref{lemma 4.2} and the boundary conditions of $R_{ij},$ we have that $\tilde{R}_{ij}\in H_{2+\sigma}^{(-1-\sigma)}(V),$ which means that  $\tilde{R}_{ij}\in C^{1,\sigma}(V)$ and $R_{ij}\in C^\sigma(V).$ We could also prove that   $\tilde{R}_{00}\in C^{1,\sigma}(V)$ and $R_{00}\in C^\sigma(V)$ in the same way.
\par To study the regularity of $R_{0i},\ i=1,2,3,$ we need to consider the following 12 equations:
 \begin{equation}\label{4.6}
\left\{\begin{array}
     {l}
    \Delta u_{0i}^\theta=f_{0i}^\theta \ \ \ in \ V\\
     N(u_{0i}^\theta)=-(g^{00})^{-\frac{1}{2}}g^{j\theta}R_{ji}+P^\theta_i(\partial g) \ \ \ on \ D
     \end{array}\right.
 \end{equation}
 Here $\theta=0,1,2,3$ and $P^\theta_i(\partial g)$ is a polynomial of $g$ and $g^{-1}$ to be determined, hence in $C^\sigma(D).$
 Lemma \ref{lemma 4.3} tells us that these 12 equations have solutions $u_{0i}^\theta\in H_{2+\sigma,}^{(-1-\sigma)}(\overline{M}).$ Now let  $\tilde{R}_{0i}=R_{0i}-\partial_\theta u_{0i}^\theta,$then
$$\Delta\tilde{R}ic_{0i}=\mathcal{Q}+Q(g,\partial g,\partial^2g,\partial u,\partial^2u)\in H_{\sigma}^{(1-\sigma)}(V).$$
We recall the Neumann boundary condition  (\ref{3.5}):
$$N(R_{0i})=(g^{00})^{-\frac{1}{2}}(-g^{j\beta}\partial_\beta R_{ji}+g^{\eta\beta}\Gamma_{i\beta}^\tau R_{\eta\tau}).$$
Then $$\begin{aligned}
N(\tilde{R}_{0i})&=N(R_{0i}-\partial_\theta u_{0i}^\theta)=N(R_{0i})-N(\partial_\theta u_{0i}^\theta)\\
                 &=(g^{00})^{-\frac{1}{2}}(-g^{j\beta}\partial_\beta R_{ji}+g^{\eta\beta}\Gamma_{i\beta}^\tau R_{\eta\tau})+\partial_\theta((g^{00})^{-\frac{1}{2}}g^{j\theta}R_{ji}-P^\theta_i(\partial g))+Q(\partial g,\partial u)\\
                 &=(g^{00})^{-\frac{1}{2}}g^{\eta\beta}\Gamma_{i\beta}^\tau R_{\eta\tau}-\partial_\theta P^\theta_i(\partial g)+Q(\partial g,\partial u,R_{ij}).
\end{aligned}$$
So if we select some good polynomial $P^\theta_i(\partial g),$ we could make that there is no second derivative of metric $g$ in $(g^{00})^{-\frac{1}{2}}g^{\eta\beta}\Gamma_{i\beta}^\tau R_{\eta\tau}-\partial_\theta P^\theta_i(\partial g).$ In other words,
$$N(\tilde{R}_{0i})=Q(\partial g,\partial u,R_{ij})\in C^{\sigma}(D)$$
We again use lemma \ref{lemma 4.3} to conclude that  $\tilde{R}_{0i}\in H_{2+\sigma}^{(-1-\sigma)}(V).$ So $\tilde{R}_{0i}\in C^{1,\sigma}(V)$ and $R_{0i}\in C^\sigma(V).$
\par Now we have proved that $R_{\alpha\beta}\in C^\sigma(V)$ for all $0 \leq\alpha,\beta\leq 3,$ so $f_{\alpha\beta}^\theta=\Gamma\ast Ric\in C^\sigma(V).$ Then the solutions of equation \ref{4.5} and \ref{4.6} $u_{\alpha\beta}^\theta$ are in $C^{2,\sigma}(V).$ Finally, we get that $R_{\alpha\beta}\in C^{1,\sigma}(V)$ by the same method above.
\par Thus we have finished the first step of the proof, i.e. $Ric\in C^{1,\sigma}(\overline{M})$ in harmonic charts.
\subsection{The $C^{m,\alpha}$ Regularity of Metric in Harmonic Charts}
We have already shown that $g_{\alpha\beta}\in C^{1,\lambda}$ for any $\lambda\in(0,1)$ in harmonic charts, then
\begin{equation} \label{4.7}
\Delta g_{\alpha\beta}=-2R_{\alpha\beta}+Q(g,\partial g)
\end{equation}
If $1\leq i,j\leq 3,$ we have the boundary conditions
$$g_{ij}=h_{ij},$$
So $g_{ij}\in C^{2,\lambda}.$
\par Let $A_{ij}$ be the second fundamental form,
$$A_{ij}=\frac{1}{2}(g^{00})^{\frac{1}{2}}g^{0\beta}(\partial_\beta g_{ij}-\partial_i g_{\beta j}-\partial_j g_{\beta i}).$$
Since $Ric\in C^{1,\sigma}(\overline{M}),$ according to (\ref{3.23}), $A_{ij}\in C^{2,\sigma}(\M).$ Combining it with that $g_{ij}\in C^{2,\lambda}(\overline{M}),$
\begin{equation} \label{4.8}
\partial_j g_{i0}+\partial_i g_{j0}\in C^{2,\sigma}(\M)
\end{equation}
Recall the boundary condition (\ref{3.24})
$$g^{\eta\beta}\partial_\eta (g_{\alpha\beta}-\frac{1}{2}\partial_\alpha g_{\eta\beta})=0.$$
Let $\alpha=0$, and with (\ref{4.8}) we conclude that
\begin{equation}\label{4.9}
(g^{j0}\partial_j+\frac{1}{2}g^{00}\partial_0)g_{00}\in C^{2,\sigma}(\M)
\end{equation}
So $g_{00}\in C^{2,\lambda}(\overline{M}).$
\par Let $\alpha=i$ in (\ref{3.24}), and with (\ref{4.8}) we get that
\begin{equation}\label{4.10}
(g^{j0}\partial_j+\frac{1}{2}g^{00}\partial_0)g_{i0}\in C^{2,\sigma}(\M)
\end{equation}
So $g_{i0}\in C^{2,\lambda}(\overline{M}).$
Now we have proved that $g$ is $C^{2,\lambda}$ in harmonic charts. Hence $\{x^\theta\}_{\theta=0}^3$ form a $C^{3,\lambda}$ differential structure of $\overline{M}.$ Repeat the steps above, we could improve the regularity of metric $g$ gradually, and finally $g\in C^{m,\alpha}(\overline{M},x).$ Hence  $\{x^\theta\}_{\theta=0}^3$ form a $C^{m+1,\alpha}$ differential structure of $\overline{M}.$

\subsection{Regularity of the Defining Function}

\par We already show that $\rho\in C^{2,\sigma}(\overline{M})$ and $\rho$ is smooth in interior. Then the only thing is to study the boundary regularity of the defining function. For any $p\in\M,$ take the harmonic chart $(V,x)$ of $p$ and let $D=V\cap\M,$ We could also assume that $g_{\alpha\alpha}=1, g_{ij}=g_{02}=g_{03}=\cdots=g_{0n}=0 (i\neq j), g_{01}=-\delta$ at $p$ where $\delta\in(0,1)$ is sufficiently close to 1. according to (\ref{2.2}) and (\ref{2.3})
$$Ric-\frac{Sg}{n+1}=-(n-1)\frac{D^2\rho}{\rho}+\frac{n-1}{n+1}\frac{\Delta\rho}{\rho}g.$$
Locally, when acting on $(\frac{\partial}{\partial x^0},\frac{\partial}{\partial x^1})$,
\begin{equation}\label{4.11}
\Delta\rho-(n+1)\cdot g_{01}^{-1}\cdot D^2\rho(\frac{\partial}{\partial x^0},\frac{\partial}{\partial x^1})=\frac{n+1}{n-1}\cdot g_{01}^{-1}\cdot\rho(Ric_{01}-\frac{Sg_{01}}{n+1})
\end{equation}
If $1-\delta$ is small enough, then the left side of the formula above is a elliptic operator around $p.$ Since $\rho|_D\equiv 0,$  $\rho\in C^{m,\alpha}(x).$
\\ In order to improve the $C^{m+1,\alpha}$ regularity of $\rho,$ we need that $\rho(Ric_{01})$ in (\ref{4.11}) is at least $C^{m-1,\alpha}.$ Actually,
$$\Delta(\rho Ric)=\rho\Delta(Ric)+Ric\Delta\rho+2g(\nabla\rho,\nabla Ric)$$
The right side of this formula is $C^{m-3,\alpha}$ with the help of Bach equation. $\rho Ric|_{\partial M}\equiv 0,$ so $\rho(Ric_{01})\in C^{m-1,\alpha},$ and the defining function $\rho$ is $C^{m+1,\alpha}.$

\bibliographystyle{plain}%

\bibliography{bibfile}

\begin{thebibliography}{10}

\bibitem{anderson2003boundary}
Michael~T Anderson.
\newblock Boundary regularity, uniqueness and non-uniqueness for ah einstein
  metrics on 4-manifolds.
\newblock {\em Advances in Mathematics}, 179(2):205--249, 2003.

\bibitem{anderson2008einstein}
Michael~T Anderson.
\newblock Einstein metrics with prescribed conformal infinity on 4-manifolds.
\newblock {\em Geometric and Functional Analysis}, 18(2):305--366, 2008.

\bibitem{andersson1996solutions}
Lars Andersson and Piotr~T Chrusciel.
\newblock Solutions of the constraint equations in general relativity
  satisfying “hyperboloidal boundary conditions”.
\newblock {\em Dissertationes Mathematicae (Rozprawy Matematyczne)},
  355:1--100, 1996.

\bibitem{besse2007einstein}
Arthur~L Besse.
\newblock {\em Einstein manifolds}.
\newblock Springer Science \& Business Media, 2007.

\bibitem{brezis1978remarks}
Haiem Brezis and Tosio Kato.
\newblock Remarks on the schroedinger operator with singular complex
  potentials.
\newblock Technical report, WISCONSIN UNIV-MADISON MATHEMATICS RESEARCH CENTER,
  1978.

\bibitem{chang2018compactness1}
Sun-Yung~A Chang and Yuxin Ge.
\newblock Compactness of conformally compact einstein manifolds in dimension 4.
\newblock {\em Advances in Mathematics}, 340:588--652, 2018.

\bibitem{chang2018compactness2}
Sun-Yung~A Chang, Yuxin Ge, and Jie Qing.
\newblock Compactness of conformally compact einstein 4-manifolds ii.
\newblock {\em arXiv preprint arXiv:1811.02112}, 2018.

\bibitem{chrusciel2005boundary}
Piotr~T Chru{\'s}ciel, Erwann Delay, John~M Lee, Dale~N Skinner, et~al.
\newblock Boundary regularity of conformally compact einstein metrics.
\newblock {\em Journal of Differential Geometry}, 69(1):111--136, 2005.

\bibitem{deturck1981some}
Dennis~M DeTurck and Jerry~L Kazdan.
\newblock Some regularity theorems in riemannian geometry.
\newblock In {\em Annales scientifiques de l'{\'E}cole Normale Sup{\'e}rieure},
  volume~14, pages 249--260, 1981.

\bibitem{gicquaud2013conformal}
Romain Gicquaud.
\newblock Conformal compactification of asymptotically locally hyperbolic
  metrics ii: weakly alh metrics.
\newblock {\em Communications in Partial Differential Equations},
  38(8):1313--1367, 2013.

\bibitem{gilbarg1980intermediate}
David Gilbarg and Lars H{\"o}rmander.
\newblock Intermediate schauder estimates.
\newblock {\em Archive for Rational Mechanics and Analysis}, 74(4):297--318,
  1980.

\bibitem{gilbarg2015elliptic}
David Gilbarg and Neil~S Trudinger.
\newblock {\em Elliptic partial differential equations of second order}.
\newblock springer, 2015.

\bibitem{graham1985conformal}
C~FEFFERMAN-CR Graham and C~Fefferman.
\newblock Conformal invariants.
\newblock {\em Elie Cartan et mathematiques d'aujourd'hui, Asterisque, hors
  serie (Societe Mathematique de France, Paris)}, pages 95--116, 1985.

\bibitem{graham1991einstein}
C~Robin Graham and John~M Lee.
\newblock Einstein metrics with prescribed conformal infinity on the ball.
\newblock {\em Advances in mathematics}, 87(2):186--225, 1991.

\bibitem{gursky2017non}
Matthew~J Gursky and Qing Han.
\newblock Non-existence of poincar{\'e}--einstein manifolds with prescribed
  conformal infinity.
\newblock {\em Geometric and Functional Analysis}, 27(4):863--879, 2017.

\bibitem{gursky2020local}
Matthew~J Gursky and G{\'a}bor Sz{\'e}kelyhidi.
\newblock A local existence result for poincar{\'e}-einstein metrics.
\newblock {\em Advances in Mathematics}, 361:106912, 2020.

\bibitem{helliwell2008boundary}
Dylan~William Helliwell.
\newblock Boundary regularity for conformally compact einstein metrics in even
  dimensions.
\newblock {\em Communications in Partial Differential Equations},
  33(5):842--880, 2008.

\bibitem{jin2019boundary}
Xiaoshang Jin.
\newblock Boundary regularity for asymptotically hyperbolic metrics with smooth
  weyl curvature.
\newblock {\em Pacific Journal of Mathematics}, 301(2):467--487, 2019.

\bibitem{kichenassamy2004conjecture}
Satyanad Kichenassamy.
\newblock On a conjecture of fefferman and graham.
\newblock {\em Advances in Mathematics}, 184(2):268--288, 2004.

\bibitem{lee1994spectrum}
John~M Lee.
\newblock The spectrum of an asymptotically hyperbolic einstein manifold.
\newblock {\em arXiv preprint dg-ga/9409003}, 1994.

\bibitem{lee2006fredholm}
John~M Lee.
\newblock {\em Fredholm operators and Einstein metrics on conformally compact
  manifolds}, volume~13.
\newblock American Mathematical Soc., 2006.

\bibitem{li1995yamabe}
MA~Li.
\newblock The yamabe problem with dirichlet data.
\newblock {\em Comptes rendus de l'Acad{\'e}mie des sciences. S{\'e}rie 1,
  Math{\'e}matique}, 320(6):709--712, 1995.

\bibitem{lieberman1986intermediate}
Gary~M Lieberman.
\newblock Intermediate schauder estimates for oblique derivative problems.
\newblock {\em Archive for Rational Mechanics and Analysis}, 93(2):129--134,
  1986.

\bibitem{maldacena1999large}
Juan Maldacena.
\newblock The large-n limit of superconformal field theories and supergravity.
\newblock {\em International journal of theoretical physics}, 38(4):1113--1133,
  1999.

\end{thebibliography}

\noindent{Xiaoshang Jin}\\
  Beijing International Center for Mathematical Research, Peking University, Beijing, P.R. China. 100871
 \\Email address:{xsjin@bicmr.pku.edu.cn}

\end{document}